\documentclass[a4paper]{amsart}
\usepackage{amsmath, amsthm,amssymb}

\newtheorem{theorem}{Theorem}[section]
\newtheorem{lemma}{Lemma}[section]
\newtheorem{proposition}{Proposition}[section]

\newtheorem{corollary}{Corollary}[section]

\newcommand{\ve}{\varepsilon}

\theoremstyle{definition}
\newtheorem{definition}{Definition}[section]

\theoremstyle{remark}
\newtheorem{remark}{Remark}[section]
\newtheorem{remarks}{Remarks}[section]
\newtheorem{example}{Example}[section]

\title{The transport speed and optimal work in pulsating Frenkel-Kontorova models}

\author{Braslav Rabar and Sini\v{s}a Slijep\v{c}evi\'{c}}

\address{ 
	Department of Mathematics, Bijeni\v{c}ka 30, Zagreb, Croatia
}

\email{brabar@math.hr,slijepce@math.hr}


\begin{document}

\begin{abstract}
We consider a generalized one-dimensional chain in a periodic potential (the Frenkel-Kontorova model), with dissipative, pulsating (or ratchet) dynamics as a model of transport when the average force on the system is zero. We find lower bounds on the transport speed under mild assumptions on the asymmetry and steepness of the site potential.  Physically relevant applications include explicit estimates of the pulse frequencies and mean spacings for which the transport is non-zero, and more specifically the pulse frequencies which maximize work. The bounds explicitly depend on the pulse period and subtle number-theoretical properties of the mean spacing. The main tool is the study of time evolution of spatially invariant measures in the space of measures equipped with the $L^1$-Wasserstein metric.
\end{abstract}

\maketitle

\section{Introduction.}

Our main motivation is to analyse the transport in spatially periodic systems far from equilibria, in the cases when there is no a-priori driving bias in any direction. Relevant physical examples range from molecular motors and molecular pumps, photovoltaic and photorefractive effects in materials, Josephson-Johnson arrays and many other examples (see \cite{Reimann02} for overview and references).

We consider perhaps the simplest model exhibiting collective ratchet behaviour and enabling rigorous results. The generalized Frenkel-Kontorova model \cite{Aubry83,Baesens05,Floria96} is an one-dimensional chain of particles with neighbouring sites coupled with a convex interaction potential $W$, in a periodical potential $V$. It is given by the formal Hamiltonian
$$ H(u)= \sum_{k= -\infty}^{\infty} \left( W(u_{k+1}-u_k)) - V(u_k) \right),$$
where $u \in \mathbb{R}^{\mathbb{Z}}$ is a {\it configuration} of the chain. The classical Frenkel-Kontorova model is defined with $W(x)=x^2$, $V(x)=k \cos(2 \pi x)$, $k>0$ a parameter. Our standing assumption is that $W,V$ are $C^2$, that $W$ is strictly convex (i.e.\ that $W''$ is positive and bounded away from zero), and that $V(x+1)=V(x)$ for all $x$.

We consider its dissipative (over-damped), {\it pulsating} (or {\it ratchet}) dynamics, with the pulsating potential. The dynamics is given with
\begin{equation}
\frac{d}{dt}u_{k}(t)=W'(u_{k+1}-u_k)-W'(u_k-u_{k-1}) + K(t)V'(u_k), \label{r:ratchet1}
\end{equation} 
where $K(t)$ is always assumed to be piecewise-continuous, bounded, periodic {\it pulse} with period $2\tau$.
We note here that all the results and calculations with straightforward modifications also hold in the case of the pulsating interaction, with the equations of motion
\begin{equation}
\frac{d}{dt}u_{k}(t)=K(t)\left[W'(u_{k+1}-u_k)-W'(u_k-u_{k-1})\right] + V'(u_k). \label{r:ratchet2}
\end{equation} 
Following initial results holding in more general cases, we omit the detailed estimates for (\ref{r:ratchet2}) and discuss only the case (\ref{r:ratchet1}).

Our focus is study of the existence of transport. The main result is an explicit lower bound on the transport speed $v(\rho)$. Here $\rho$ is the mean spacing of a configuration $u$ given with
\begin{equation}
\rho(u)=\lim_{n-m \rightarrow \infty}\frac{u_n-u_m}{n-m}  \label{d:rho}
\end{equation}
(whenever the limit exists) and $v(\rho)$ is defined in Theorem \ref{t:speed}.

We say that the transport exists if $v(\rho)>0$. It has already been shown analytically in \cite{Floria05} and numerically in \cite{Floria02} that transport can exist for models (\ref{r:ratchet1}), (\ref{r:ratchet2}). (Transport can also exist in similar systems with a stochastic force, not studied here - see \cite{Saakian:18,Wang:17} and references therein.) The authors showed in \cite{Floria05} that for a large class of potentials $V$ satisfying certain asymmetry condition, the transport exists in the limit $\tau \rightarrow \infty$. The used method, however, does not provide much information on the transport speed $v(\rho)$, as it considers large $\tau$ and assumes that the model completely relaxes between pulse "switches". (The approach is based on comparison of ground states of the Frenkel-Kontorova model with a given potential $V$, and with $V \equiv 0$, using the tools from Aubry-Mather theory \cite{Aubry83,Bangert88,Katok05,Mather82}.) 

Our approach is dynamical, and gives lower bounds on the transport speed as long as $K \cdot V$ is sufficiently asymmetric and steep, also in the cases when the period of the pulse $\tau$ is finite and we always remain relatively far away from equilibria. We are thus able to estimate minimal $\tau$ so that the transport speed is $>0$, and also heuristically discuss which $\tau$ maximizes $v(\rho)$. The physical meaning of this is determination of the optimal pulse frequency for a given model and mean spacing which optimizes work.

We confirm and refine findings \cite{Floria02,Floria05} that $v(\rho)$ depends non-trivially on number-theoretical properties of the mean spacing $\rho$, and show how it explicitly depends on the sequence of convergents of $\rho$. We also obtain an expansion of the lower bound for a generic mean spacing $\rho$, where generic are those (irrational) $\rho$ for which a version of the central limit theorem for its sequence of continued fraction approximations holds (the Khinchin-L\'{e}vy Theorem \cite{Khinchin97}).

We develop the required theory in several steps. The results rely on rather general results for non-autonomous, periodically driven over-damped Frenkel-Kontorova dynamics, following the approach as in \cite{Baesens98,Qin11,Qin13,Slijepcevic14,Slijepcevic14b}. The lower bound on the transport speed is then obtained in three steps, stated precisely in Section \ref{s:outline}, and then proved in three dedicated sections of the paper. The first two steps result with explicit lower bounds for dynamics in the off- and on-phase, relying on observing evolution of an ensemble of configurations, or more precisely evolution of shift-invariant probability measures on the space of configurations. In the third step, we relate the $L^1$-Wasserstein distance between the distribution of the positions of the configuration sites $\mod 1$ at the end of the off-phase, and the Lebesgue measure on $\mathbb{S}^1$. The number-theoretical properties of the mean spacing emerge as important in this particular step of the calculation.

\section{Statement of the results and optimal work} \label{s:two}

Following \cite{Baesens98,Floria96,Middleton:92}, we say that a configuration $u \in \mathbb{R}^{\mathbb{Z}}$ is of bounded width, if there exists a real number $\rho \in \mathbb{R}$ and a constant $n > 0$, such that for all $j,k \in \mathbb{Z}$, $k\geq 1$, we have
\begin{equation}
|u_{j+k}-u_j - k\rho | \leq n. \label{d:width}
\end{equation}
Then $\rho$ is the mean spacing as defined in (\ref{d:rho}). Without loss of generality, we always choose $n$ to be an integer. We first establish the following general result which introduces the transport speed:
\begin{theorem} \label{t:speed}
	There exists a unique, continuous function $v : \mathbb{R} \rightarrow \mathbb{R}$, called the transport speed, depending on $W,V,K$, such that for any $u^0 \in \mathbb{R}^{\mathbb{Z}}$ of bounded width with the mean spacing $\rho$ and constant $n$ as in (\ref{d:width}), for any $k,m \in \mathbb{Z}$, $m \geq 1$ and any $t_0 \in \mathbb{R}$, if $u(t)$ is the solution of (\ref{r:ratchet1}) or (\ref{r:ratchet2}), $u(t_0)=u^0$, then
	\begin{equation}
	  \left| \frac{u_k(t_0+2m\tau)-u_k(t_0)}{2m\tau} - v(\rho) \right| \leq  \frac{n+2}{2m\tau}. \label{r:five}
	\end{equation}
\end{theorem}

We prove it in Section \ref{s:four}, and give an alternative characterization of transport speed as a mean displacement over one period of the pulse, averaged against a certain measure on configurations corresponding to dynamic ground states. Note that relation (\ref{r:five}) gives an explicit estimate on how quickly the average displacement converges to the transport speed, thus for example useful in numerical simulations.

The main result on the transport speed in the pulsating potential case (\ref{r:ratchet1}) is shown for the following step-function pulse:
let $\kappa \geq 0$, and let $K$ be given with
\begin{equation}
K(t) = \begin{cases} 0, & 2n \tau \leq t < (2n+1) \tau, \: n \in \mathbb{Z}, \\
\kappa & (2n+1) \tau \leq t < (2n+2) \tau, \: n \in \mathbb{Z}.
\end{cases} \label{r:a1}
\end{equation}
Furthermore, assume that there exist $\beta > 0$ and $0 < \alpha < 1/2$ so that for some $a,b$, $1/2+\alpha \leq b-a <1$  and for all  $x \in [a,b]$,  we have
\begin{equation}
\kappa \cdot V'(x) \geq \ \delta^+ + \beta, \label{r:a2}
\end{equation}
where
\begin{equation}
\delta^- = \min \lbrace W''(x), \: x \in [\rho-1,\rho+1] \rbrace, \hspace{5ex} \delta^+ = \max \lbrace W''(x), \: x \in [\rho-1,\rho+1] \rbrace.  \label{d:deltadef}
\end{equation}
The condition (\ref{r:a2}) is the required "steepness and asymmetry" assumption. The physical interpretation of (\ref{r:a2}) is as follows: on an interval of the length at least $\alpha$ larger than the spatial period (normalized to $1$), the potential force on a particle in the on-phase dominates the particle interaction force by at least $\beta$.

We introduce a function depending on number-theoretical properties of $\rho$ and the period $\tau$:
\begin{align}
   d(\rho,\tau) = \inf \left(
    \frac{1}{4q} + \frac{q}{4}|p/q-\rho| + \frac{q}{\sqrt{12}}\cdot \frac{\delta^+}{\left( 2(\delta^-)^3\tau+\delta^+ \delta^- \right)^{1/2}} 
   \right) \wedge 1/4,   \label{d:d}
\end{align}
where infimum goes over all the rationals $p/q \in \mathbb{Q}$, $p,q$ relatively prime, $q \geq 1$  (the symbols $\vee$, $\wedge$ always denote the maximum, respectively minimum of two values). We will see later that $d(\rho,\tau)$ gives an upper bound on a certain distance of the distribution of positions of $u_n$ at the end of the on-phase, and the uniform distribution. The main result of the paper is as follows (we write $v(\rho,\tau)$ to emphasize its dependence on $\tau$):
\begin{theorem} \label{t:main1}
	Assuming (\ref{r:a1}), (\ref{r:a2}), we have
	 \begin{equation}
	 v(\rho,\tau) \geq \frac{1}{\tau} \left( \frac{\alpha}{2} - d(\rho,\tau)^{1/2} +  d(\rho,\tau) -
	  \frac{1}{4} \left[ (1/2 + \alpha -\beta\tau - 2d(\rho,\tau)^{1/2}) \vee 0 \right]^2
	  \right). \label{r:mainpotential}
	 \end{equation}
\end{theorem} 

\begin{remark} \label{r:tilded}
	One can estimate $d(\rho,\tau)$ and the lower bound (\ref{r:mainpotential}) in the following way: it suffices to insert only the series of convergents $p_n/q_n$ of $\rho$ on the right-hand side of (\ref{d:d}) (they can be obtained by calculating continued fractions \cite{Khinchin97}), and thus get
	\begin{equation}
	\tilde{d}(\rho,\tau) = \min \left(
	\frac{1}{2q_n} + \frac{q_n}{\sqrt{12}}\cdot \frac{\delta^+}{\left( 2(\delta^-)^3\tau+\delta^+ \delta^- \right)^{1/2}}  
	\right) \wedge 1/4, \label{r:td}
	\end{equation} 
as for a convergent, $|p_n/q_n-\rho| \leq 1 / q_n^2$. As the right-hand side of (\ref{r:mainpotential}) is decreasing in $d(\rho,\tau)$, we obtain another lower bound by inserting $\tilde{d}(\rho,\tau)$ instead of $d(\rho,\tau)$ in (\ref{r:mainpotential}).
\end{remark}

We say that transport exists if $v(\rho,\tau) > 0$. The estimate (\ref{r:mainpotential}) allows us to deduce the asymptotic behaviour of $v(\rho,\tau)$ as follows:
\begin{corollary} \label{c:c} Assume (\ref{r:a1}), (\ref{r:a2}). Then:
	
(i) For rational $\rho=p_0/q_0$, $p_0,q_0$ relatively prime, $q_0 \geq 2$, if 
$$ \alpha > \frac{1}{\sqrt{q_0}}-\frac{1}{2q_0}, $$
then there exists $\tau_0$ such that for all $\tau > \tau_0$, transport exists, and
$$
\liminf_{\tau \rightarrow \infty} \frac{2v(\rho,\tau)\tau}{\alpha - 1/\sqrt{q_0}+1/(2q_0)} \geq 1.
$$

(ii) For all irrational $\rho$, transport exists for sufficiently large $\tau$, and
\begin{equation}
v(\rho,\tau_n) \geq \frac{\alpha}{2}\tau_n^{-1} - c \tau_n^{-9/8} + \mathcal{O}\left(\tau_n^{-5/4}\right)
\end{equation}
as $\tau \rightarrow \infty$, where $c= (\delta^+)^{1/4} / \left( 6^{1/8} (\delta^-)^{3/8} \right)$. 

(iii) There exists a subset $R \subset \mathbb{R}$ of full Lebesgue measure, such that for all $\rho \in R$ and all $\varepsilon >0$,
\begin{equation}
v(\rho,\tau) \geq \frac{\alpha}{2}\tau^{-1} - 
c\left( \frac{\gamma_L+1}{2} \right)^{1/2}(1+\varepsilon)\tau^{-9/8}  + \mathcal{O}\left(\tau^{-5/4} \right), \label{r:expansion}
\end{equation}
as $\tau \rightarrow \infty$.
\end{corollary}
Here $ \gamma_L=\exp (\pi^2 / (12 \ln 2))$ is the Khinchin-L\'{e}vy constant \cite{Khinchin97}, and the implicit constant in (\ref{r:expansion}) depends on $\varepsilon, \rho$ (it can be explicitly calculated for large subsets of $\mathbb{R}$, see Remark \ref{r:last} for details).

\begin{example} \label{e} We give a simple example of an application of Theorem \ref{t:main1} and existence of transport for $\tau = 5000$. Let
	$	W(x)=x^2/2$ 
and $V$ be any $C^2$, $1$-periodic function so that on the interval $[0,0.9]$ we have $V'(x) \geq 2$.
Then $\delta^+=\delta^-=1$ and we can choose $\alpha = 0.4$, $\beta =1$ so that (\ref{d:deltadef}) holds. By inserting this and $q_n=13$ in (\ref{r:td}) we obtain $\tilde{d}(\rho,\tau) \lessapprox 0.07598$, thus $-d(\rho,\tau)^{1/2}+d(\rho,\tau) \gtrapprox 0.19967$ for all $\rho$ with $p/13$, $p$ relatively prime with $13$, in its sequence of convergents. As $\alpha/2 = 0.2$, and the last term in (\ref{r:mainpotential}) is dominated by $\beta \tau$ thus $0$, Theorem \ref{t:main1} implies that $v(\rho,\tau) \geq (0.2-0.19967)/5000 \approx 6.6 \cdot 10^{-8}$.

To obtain $v(\rho) > 0$ for smaller $\tau$ we can choose $	W(x)=\delta x^2 /2$ and $ \tau = 5000 / \delta$ as long as $\tau \geq 1$ (so that the last term in (\ref{r:mainpotential}) is $0$) with all the other assumptions the same, as $d(\rho,\tau)$ depends on $\delta \tau$ only.

We discuss in the last section inefficiencies and potential ways to improve the bound (\ref{r:mainpotential}).
\end{example}

\begin{remarks}
(i) Consider the case when $\rho=m$ is an integer. It is known \cite{Floria05} that the transport can not exist	in that case, and can be proved directly by considering dynamics of configurations satisfying $u_k=u_{k-1}+m$ for all $k \in \mathbb{Z}$. Our results are consistent, as by (\ref{d:d}), $d(1,\tau)=1/4$ for all $\tau>0$. As $\alpha <1/2$, the right-hand side of (\ref{r:mainpotential}) is always $<0$.

(ii) The condition in \ref{c:c}, (i) seems close to optimal, as transport can exist for $\alpha$ sufficiently close to $1/2$ whenever $q_0 \geq 2$, which is consistent with findings in \cite{Floria05}. Also, our finding that larger asymmetry $\alpha$ implies existence of transport for $\tau$ large enough for rationals with smaller denominators is consistent with \cite{Floria05}. 

(iii) Existence of transport for irrational $\rho$ and sufficiently large $\tau$ also follow from \cite{Baesens05}. We note that our explicit lower bound on $v(\rho)$ and the expansion of the lower bound on $v(\rho)$ are new.
	
(v) The set $R$ in Corollary \ref{c:c}, (iii) is the set of all the numbers for which the Khinchin-L\'{e}vy Theorem \cite{Khinchin97} on asymptotic behaviour of convergents holds, see Section \ref{s:last}.

(vi) If asymptotic behaviour of the sequence of denominators $q_n$ of convergents of $\rho$ is known, one can evaluate the expansion of the lower bound of $v(\rho)$ analogously to (\ref{r:expansion}). For example, consider the golden mean $\rho_{\text{GM}} = (1 + \sqrt{5})/2$. Then $q_n = ((1+\sqrt{5})/2)^n / \sqrt{5}+ \varepsilon_n$, $\varepsilon_n < 1$, is the $n$-th Fibonacci number. Repeating the steps as in the proof of Corollary \ref{c:c}, (iii) in Section \ref{s:last}, one obtains
$$
v(\rho_{\text{GM}}) \geq \frac{\alpha}{2}\tau^{-1} -  c   \left(\frac{3+\sqrt{5}}{4}\right)^{1/2}(1+\varepsilon) \tau^{-9/8} + \mathcal{O}\left(\tau^{-5/4} \right)
$$
for all $\varepsilon>0$. Note that for $\rho=\rho_{\text{GM}}$, the Khinchin-L\'{e}vy Theorem does not hold, thus the expansion differs from (\ref{r:expansion}).
\end{remarks}

Finally, we comment on the frequency $1/\tau_{\text{max}}$ of the pulse which maximizes work, i.e. which maximizes speed $v(\rho)$, in cases for which transport exists for a given $\rho$ and $\tau$ large enough. It follows from \cite{Baesens05} that $v(\rho) \rightarrow 0$ as $\tau \rightarrow 0$ and as $\tau \rightarrow \infty$, thus $v(\rho)$ has a maximum for a finite $\tau$. By differentiating the first two terms on the right-hand side of (\ref{r:expansion}), we obtain that the pulsating frequency which maximizes work for a generic $\rho \in \mathbb{R}$ is for small $\alpha >0$ likely of the order of magnitude
$$
 \frac{1}{\tau_{\text{max}}} \sim c^{-8}\alpha^8,
$$
where $c$ is as in Corollary \ref{c:c}. It would be interesting to investigate it further either numerically, or by a more precise analytical method.

\section{Outline of the paper} \label{s:outline}

We introduce here the main definitions and the notation, and explain the three steps to obtain the main result (\ref{r:mainpotential}), stated explicitly as propositions below.

We will always consider the dynamics of (\ref{r:ratchet1}) on the phase space of configurations $\tilde{\mathcal{X}}$ of the bounded spacing, that is the configurations $ u \in \mathbb{R}^\mathbb{Z}$ such that
$$ \sup_{k \in \mathbb{Z}} |u_k - u_{k+1}| < \infty.$$
We will denote by ${\mathcal{X}}$ the quotient space of $\tilde{\mathcal{X}}$, where we identify $u$ and $u+n$ for all integers $n$. We always assume the product topology on $\tilde{\mathcal{X}}$ (i.e.\ the topology of pointwise convergence), and the topology on $\mathcal{X}$ induced from $\tilde{\mathcal{X}}$. 

Our approach focuses on considering dynamics of particular solutions of (\ref{r:ratchet1}), namely of synchronized solutions. These are solutions $u(t)$ such that $u(t)$ does not intersect with either any spatial translate of $u(t)$, or $u(t+2n\tau)$ for any integer $n$ (see Definition \ref{d:synorbit}). More generally, we consider evolution of a particular probability measure $\mu^0$ on $\mathcal{X}$ supported on synchronized solutions. A precise definition of a synchronized measure, and a proof of its existence for any rotation number $\rho$, is given in Section \ref{s:four}. We denote this evolution by $\mu(t)$. 

The reason for this approach is as follows. In most rigorous approaches to studying dynamics of Frenkel-Kontorova models, one has to frequently spatially average certain quantities. In our approach, this is replaced by averaging these quantities with respect to $\mu(t)$. For example, if $\mu$ is a spatially ergodic probability measure on $\mathcal{X}$ (i.e. ergodic with respect to the spatial translation $S : \mathcal{X} \rightarrow \mathcal{X}$, $(Su)_k=u_{k-1}$), then by the Birkhoff ergodic theorem, for $\mu$-a.e. configuration $u$, its mean spacing is given with
\begin{equation}
\rho(u) = \int_{\mathcal{X}} \left( u_1-u_0 \right) d\mu. \label{d:rotation}
\end{equation}
We frequently denote the right-hand side of (\ref{d:rotation}) by $\hat{\rho}(\mu)$. (The 'hat' symbol will always denote a function whose argument is a probability measure on $\mathcal{X}$.)

\begin{remark} {\bf On integration on $\mathcal{X}$}. As $\mathcal{X}$ is the quotient space of configurations, one must verify that the integrals such as in (\ref{d:rotation}) are well-defined, i.e. that the integrand is independent of the choice of the integer $n$ in the representation $u+n$, $u \in \tilde{\mathcal{X}}$. We frequently use the fact that the following functions are well defined on $\mathcal{X}$ for all $k \in \mathbb{Z}$: $u_k \mod 1$, $u_k-u_{k-1}$, $u_{k+1}-2u_k+u_{k-1}$, $V'(u_k)$, $du_k(t)/dt$. In addition, 
$u_k(t_2)-u_k(t_1)$ is also well defined for any $k \in \mathbb{Z}$, $t_1,t_2 \geq0$, as we always assume choosing consistent representations of $u(t)$, i.e.\ such that that the representation of $u(t_2)$ is the time-evolution of the representation of $u(t_1)$.
	
\end{remark}

We will see that in this setting, it suffices to consider the average (temporal) displacement in the off-phase, which can be expressed as
$$
D_{\text{off}}(\rho) = \int_{\mathcal{X}} \left(u_0(\tau)-u_0(0) \right) d\mu^0,
$$
and in the on-phase,
$$
D_{\text{on}}(\rho) = \int_{\mathcal{X}} \left(u_0(2 \tau)-u_0(\tau) \right) d\mu^0.
$$

In Sections \ref{s:three} and \ref{s:four} we prove the required results on synchronized solutions and synchronized measures, prove Theorem \ref{t:speed}, and establish that
\begin{equation} \label{r:vformula}
v(\rho) = \frac{1}{2\tau} \left( D_{\text{off}}(\rho) + D_{\text{on}}(\rho) \right).
\end{equation}

While the mean spacing is the average (expectation) of the spatial displacement $u_1-u_0$, we will also find the standard deviation of the spatial displacement useful. We denote it by
\begin{equation} 
\hat{\Delta}(\mu)= \left( \int_{\mathcal{X}} (u_1-u_0-\hat{\rho}(\mu))^2 d\mu \right)^{1/2}. \label{r:off}
\end{equation} 
In Section \ref{s:five}, we consider dynamics in the off-phase, and prove the following:
\begin{proposition} \label{p:first} We have that 
\begin{equation}
D_{\text{off}}(\rho) = 0, \label{r:don}
\end{equation}
and 
\begin{equation}
\hat{\Delta}(\mu(\tau)) \leq\frac{\delta^+}{\left( 2(\delta^-)^3\tau+\delta^+ \delta^- \right)^{1/2}}, \label{r:widthestimate}
\end{equation}
where $\delta^-$, $\delta^+$ are as in (\ref{d:deltadef}).
\end{proposition}

For example, in the standard case $W(x)=\frac{1}{2}\delta x^2$, (\ref{r:widthestimate}) becomes $\hat{\Delta}(\mu(\tau)) \leq 1/(2\delta \tau+1)^{1/2}$. 

To obtain the lower bound on $D_{\text{on}}(\rho)$, we introduce two tools. Let $\pi : \mathcal{X} \rightarrow \mathbb{S}^1$ be a projection defined with $\pi(u)= u_0 \text{ mod }1$. Let $\pi^*$ be the associated pushforward of a measure $\mu$ on $\mathcal{X}$ to a measure on $\mathbb{S}^1$. We always use the notation $\mu^*= \pi^*\mu$ for the pushforwarded measure on $\mathbb{S}^1$ (it is well-defined by the definition of $\mathcal{X}$ as a quotient space). Note that, as $\mu$ is $S$-invariant, the definition of the measure $\mu^*$ is independent of the choice of the coordinate at which we project to $\mathbb{S}^1$. (For a comment on close relationship of the projections $\mu(t)^*$ and the dynamical hull function introduced by Qin \cite{Qin13}, see Remark \ref{rm:hull}.)

The second tool we need is the $L^1$-Wasserstein distance on the space of Borel probability measures on $\mathbb{S}^1$. Recall the definition \cite{Ambrosio05, Villani:08}: if $\mu^*$, $\nu^*$ are two Borel-probability measures on $\mathbb{S}^1$, then the $L^1$-Wasserstein distance is defined with
$$ d_1(\mu^*,\nu^*)= \inf_{\zeta \in \Gamma(\mu^*,\nu^*)} \int_{\mathbb{S}^1 \times \mathbb{S}^1} d(x,y)d\zeta(x,y),$$
where $\Gamma(\mu^*,\nu^*)$ is the set of couplings, i.e.\ Borel probability measures on $\mathbb{S}^1 \times \mathbb{S}^1$ whose marginals are $\mu^*$, $\nu^*$ respectively. Now $d_1$ is a metric on the space of Borel-probability measures on $\mathbb{S}^1$.

We show in Section \ref{s:seven} that an upper bound on $d_1(\mu(\tau)^*,\lambda)$ implies a lower bound on $D_{\text{on}}(\rho)$, where $\lambda$ is the Lebesgue measure on $\mathbb{S}^1$:

\begin{proposition} \label{p:second}
Assume that $d_1(\mu(\tau)^*,\lambda) \leq \varepsilon \leq 1/4$. Then
	\begin{equation}
	D_{\text{on}}(\rho) \geq  \dfrac{1}{\tau}\left( \frac{\alpha}{2} -  \varepsilon^{1/2}+\varepsilon -\frac{1}{4}\left[ \left( \frac{1}{2} + \alpha - \beta \tau - 2\varepsilon^{1/2} \right) \vee 0 \right]^2 \right). \label{r:formula}
	\end{equation}
\end{proposition}
Finally, it remains to relate $\hat{\Delta}(\mu(\tau))$ and $d_1(\mu(\tau)^*,\lambda)$. The number-theoretical properties of $\rho$ enter in this step. This is done in Section \ref{s:seven}.

\begin{proposition} \label{p:third} Let $\mu$ be a $S$-invariant probability measure on $\mathcal{X}$. Then for any $p/q \in \mathbb{Q}$, $q \geq 1$, $p,q$ relatively prime, we have that
\begin{equation}
  d_1(\mu^*,\lambda) \leq \left( \frac{1}{4q} + \frac{q}{4}|p/q-\rho| + \frac{q}{\sqrt{12}} \hat{\Delta}(\mu) \right) \wedge \frac{1}{4}.
\end{equation}
\end{proposition}

Theorem \ref{t:main1} now follows easily by combining these three propositions while noting that the right-hand side of (\ref{r:formula}) is decreasing in $\varepsilon$ for $0 \leq \varepsilon \leq 1/4$, and then inserting (\ref{r:don}) and (\ref{r:formula}) into (\ref{r:vformula}). Finally, Corollary \ref{c:c} is proved in Section \ref{s:last}.

\section{Existence of synchronized orbits} \label{s:three}

In this section we focus on existence of synchronized solutions, which play the key role in studying asymptotic behaviour of (\ref{r:ratchet1}), (\ref{r:ratchet2}). The main result is that synchronized solutions with any mean spacing $\rho$ exist. The proof closely follows Wen-Xin Qin \cite{Qin13}, whose approach is extended here to the following family of equations which include (\ref{r:ratchet1}), (\ref{r:ratchet2}):
\begin{equation}
\label{maineq}
\frac{d}{dt}u_{j}(t)=-V_{2}(u_{j-1}(t),u_{j}(t),t)-V_{1}(u_{j}(t),u_{j+1}(t),t), \hspace{3ex} j \in \mathbb{Z}
\end{equation}
($V_1,V_2$ denote the partial derivatives with respect to the first, resp. second coordinate). We assume that $V \colon \mathbb{R}^{3}\to \mathbb{R}$ satisfies the following in this and the next section: 
\begin{itemize}
	\item[(A1)] Smoothness: $(t,u,v)\mapsto V(t,u,v)$ is $C^{2}$ in $u,v$, and bounded and measurable in $t$ for fixed $u,v$.
	\item[(A2)] Periodicity: there is a fixed $\tau>0$ such that for all $u,v,t$ we have $V(u,v,t)=V(u+1,v+1,t)=V(u,v,t+2\tau)$,
	\item[(A3)] The twist condition: for some fixed $\delta > 0$ and all $u,v,t$,
	\[
	\frac{V(u,v,t)}{\partial u \partial v} \leq -\delta <0.
	\]
	\item[(A4)] Boundedness: There exists $C>0$ such that $\|D_{u,v}^{2}V(u,v,t)\|\leq C$. 
\end{itemize}

Let $T$ be the time-$2 \tau$ map $Tu(0)=u(2\tau)$, where $u(t)$ is the solution of (\ref{maineq}) (it follows from Lemma \ref{l:semiflow}, (i) and (v) below that $T$ is well-defined and that $T^nu(0)=u(2n\tau)$ for integer $n$). The shift $S$ and $T$ are then commuting continuous transformations on both $\tilde{\mathcal{X}}$ and $\mathcal{X}$. We say that $u,v \in \tilde{\mathcal{X}}$ do not intersect, if either $u=v$, $u \gg v$ or $u \ll v$, where the partial order on $\tilde{\mathcal{X}}$ is defined in the natural way: we write $u \geq v$ if $u_n \geq v_n$ for all $n\in \mathbb{Z}$, and $u \gg v$ if $u_n > v_n$ for all $n \in \mathbb{Z}$. 
We say that $u,v \in \mathcal{X}$ do not intersect, if their representations in $\tilde{\mathcal{X}}$ do not intersect for any choice of the representations. We recall the definition of rotationally ordered solutions, standard in the study of Frenkel-Kontorova models \cite{Baesens05,Baesens98,Floria02,Floria05,Floria95,Floria96,Hu05,Qin10,Qin11}, and of synchronized solutions \cite{Qin13,Slijepcevic14,Slijepcevic14b}:

\begin{definition} \label{d:synorbit} We say that $u \in \mathcal{X}$ is {\it rotationally ordered}, if for any $n\in \mathbb{Z}$, we have that $u$ and $S^n u$ do not intersect. A configuration $u \in \mathcal{X}$ is {\it weakly rotationally ordered}, if it is in the closure of the set of rotationally ordered configurations. A solution $t \mapsto u(t) \in \mathcal{X}$, $t \in \mathbb{R}$ of (\ref{maineq}) is {\it synchronized solution}, if for any $s \in \mathbb{R}$ and $n,m \in \mathbb{Z}$, we have that $u(s)$ and $S^n T^m u(s)$ do not intersect. A configuration $u^0 \in \mathcal{X}$ is synchronized, if there exists a synchronized solution of (\ref{maineq}), $u(0)=u^0$.
\end{definition}

The main result is an analogue of \cite[Theorems 3.4, 3.5]{Qin13}:
\begin{theorem}
	\label{t:existence}
	Given any mean spacing $\rho \in \mathbb{R}$, there exists a solution $u(t)$, $t\in \mathbb{R}$ of (\ref{maineq}) which is synchronized, such that for all $t \in \mathbb{R}$, $u(t)$ has the mean spacing $\rho$.
\end{theorem}
Following Qin \cite{Qin13}, we prove it in several steps, omitting proofs in cases the extension to (\ref{maineq}) is straightforward. We first establish global existence of solutions generated by (\ref{maineq}),  then for a rational $p/q$ construct an appropriate space of functions $\Omega$ and a continuous operator $\tilde{T}$, and finally prove Theorem \ref{t:existence} for rational $\rho=p/q$ by finding a fixed point of $\tilde{T}$ through an application of the Schauder fixed point theorem. The proof is completed for irrational $\rho$ by a limiting procedure. Recall \cite{Qin13,Slijepcevic13} that $\tilde{\mathcal{X}}$, $\mathcal{X}$ are metrizable, and that the sets $\mathcal{K}_n \subset \mathcal{X}$, $n \in \mathbb{N}$ of all $u \in \mathcal{X}$ satisfying $ \sup_{k \in \mathbb{Z}} |u_k - u_{k+1}| \leq n$ are compact.  We will also frequently use the fact \cite{Baesens98,Qin13} that any weakly rotationally ordered configuration $u$ is of bounded width with the constant $1$, i.e.\ that it has a well-defined mean spacing $\rho$, and that for all $m,k \in \mathbb{Z}$, $k \geq 1$,  we have
\begin{equation}
|u_{m+k}-u_m-k\rho| \leq 1. \label{r:width1}
\end{equation}

\begin{lemma} \label{l:semiflow}
	(i) For any $u^0 \in \mathcal{X}$, the solution of (\ref{maineq}), $u(0)=u^0$ exists for all $t>0$, and is continuous with respect to initial conditions.	
	
	(ii) The solution $t \rightarrow u(t)-u(0)$ is continuous in $l^{\infty}(\mathbb{Z})$.
	
	(iii) The sets $\mathcal{K}_n$ are positively invariant with respect to (\ref{maineq}), i.e. if $u(s) \in {\mathcal K}_n$, then for all $t \geq s$, $u(t) \in {\mathcal K}_n$.
	
	(iv) The solutions of (\ref{maineq}) are strongly order-preserving, i.e.\ if $u(0) \geq v(0)$ but not equal, then for all $t >0$, $u(t) \gg v(t)$,
	
	(v) If $u(t),v(t)$ are solutions of (\ref{maineq}), $u(0)=u^0$, resp. $v(2n \tau)=u^0$, $n$ an integer, then $v(t+2n\tau)=u(t)$.
\end{lemma}

\begin{proof}
	Local existence and uniqueness follow from standard results on existence of solutions of ODE in the Banach space $l^{\infty}(\mathbb{Z})$ \cite{Rabar15,Slijepcevic13}, applied to $v = u-u^0$. Continuity with respect to initial conditions in the topology of $\mathcal{X}$ follows from the variation of constants formula. The order-preserving property (iv), also called  Middleton's no-passing rule \cite{Middleton:92} follows as in \cite{Baesens05}, as the off-diagonal elements of the linearisation of the right-hand side of (\ref{maineq}) are strictly positive and bounded from below. The proof (iii) follows easily from the no-passing rule (details can be found e.g. in \cite{Slijepcevic13}), and this implies global existence of solutions as the solution can not diverge in finite time. Proof of (v) is a straightforward consequence of (A2).
\end{proof}

For a rational $p/q$, $p,q$ relatively prime integers, $p>0$, let $\tilde{\mathcal{X}}_{p,q}$ be the set of all $u \in \tilde{\mathcal{X}}$ such that $S^qu=u - p$. It is convenient to introduce a new norm on $\tilde{\mathcal{X}}_{p,q}$ with
$$||u||_{\tilde{\mathcal{X}}_{p,q}}=\sum_{k=0}^{q-1} |u_k|$$
($\tilde{\mathcal{X}}_{p,q}$ equipped with that norm is continuously embedded in $\tilde{\mathcal{X}}$).

\begin{definition} For a fixed rational $p/q$, the set $\Omega$ is a subset of $C^0(\mathbb{R},\tilde{\mathcal{X}}_{p,q})$ of $h : \mathbb{R} \rightarrow \tilde{\mathcal{X}}_{p,q}$ satisfying:
\begin{itemize}
 \item[(O1)] $h$ is increasing, i.e.\ if $s_1 > s_2$, then $h(s_1) \geq h(s_2)$,

\item[(O2)] The image of $h$ consists of weakly rotationally ordered configurations,

\item[(O3)] For all $s \in \mathbb{R}$,
$$ s=\frac{1}{q} \sum_{k=0}^{q-1} (h(s))_k,$$

\item[(O4)] For any $m,n \in \mathbb{Z}$, the image of $h$ is invariant for $ u \mapsto S^mu+n$.
\end{itemize}
\end{definition}

The topology on $\Omega$ induced by $C^0(\mathbb{R},\tilde{\mathcal{X}}_{p,q})$ is metrizable, and the metric can be given explicitly e.g. by the norm \cite[p3480]{Qin13}.

\begin{lemma} \label{l:compact}
	The set $\Omega$ with the induced $C^0(\mathbb{R},\tilde{\mathcal{X}}_{p,q})$ topology is non-empty, convex and compact.
\end{lemma}

\begin{proof} $\Omega$ is non-empty, as for example $h$ defined with $h(s)_k=s+kp/q-(q-1)p/2$ is in $\Omega$. Convexity is analogous to \cite[Lemma 3.2]{Qin13}, and compactness follows from an application of the Arzel\`{a}-Ascoli theorem as in \cite[Lemma 3.3]{Qin13}.
\end{proof}

\begin{lemma} \label{l:operator}
	There exists a continuous function $\xi : \mathbb{R} \rightarrow \mathbb{R}$ such that the operator $\tilde{T} : \Omega \rightarrow \Omega$ defined with $\tilde{T}h (s)= T \circ h \circ \xi(s)$ is well-defined and continuous.
\end{lemma}

\begin{proof}
	The function $\xi$ is constructed as in the proof of \cite[Theorem 3.4]{Qin13}.
\end{proof}

\begin{remark}
	The function $\xi$ is unique, constructed so that (O3) holds for the image of $\tilde{T}$.
\end{remark}

\begin{proof}[Proof of Theorem \ref{t:existence}]
For a rational $\rho = p/q$, we can by Lemmas \ref{l:compact} and \ref{l:operator} apply the Schauder Fixed Point theorem to $\tilde{T}$ on $\Omega$, and find a fixed point $h_{p,q}$ of $\tilde{T}$. By construction and Lemma \ref{l:semiflow}, (iv), for any $s_0 \in \mathbb{R}$, the solution $u(t)$ of (\ref{maineq}), $u(s_0)=h_{p,q}(s_0)$, exists for all $t \in \mathbb{R}$ and is synchronized. We obtain the claim for irrational $\rho$ analogously as in the proof of \cite[Theorem 3.5]{Qin13}, by finding a limit of the sequence of images of $h_{p_k/q_k}$ in the Hausdorff topology as $p_k/q_k \rightarrow \rho$. 
\end{proof}

\section{The transport speed} \label{s:four}

In this section we prove Theorem \ref{t:speed}. We first introduce the notion of the dynamic ground state, and more generally we consider probability measures invariant with respect to evolution given by (\ref{maineq}) and spatial translations. We then establish existence of synchronized measures, i.e.\ $S,T$-invariant measures supported on synchronized configurations, and then give a characterization of the transport speed as expectation of the displacement with respect to a synchronized measure. We use it to give a short proof of Theorem \ref{t:speed}, as well as to establish a characterization of the existence of transport in terms of evolution of the dynamic ground states.

\begin{lemma} 
For each $\rho \in \mathbb{R}$, there exists a $S,T$-invariant measure supported on synchronized configurations with the mean spacing $\rho$.
\end{lemma}

\begin{proof}
It is straightforward to show by using Lemma \ref{l:semiflow} and (\ref{d:width}) that the set of synchronized configurations with the mean spacing $\rho$ is a compact invariant subset of $\mathcal{X}$, thus by a straightforward extension of the Krylov-Bogolyubov argument to commuting continuous transformations $S,T$ on compact metric spaces, there exists a $S,T$-invariant measure supported on that set. (For example, it is any limit point in the weak$^*$-topology on the set of probability measures on $\mathcal{X}$, of the sequence of measures 
$$\mu_n = \frac{1}{n^2} \sum_{j=0}^{n-1}\sum_{k=0}^{n-1}\delta_{S^jT^ku},$$ 
where $u$ is any synchronized configuration with the mean spacing $\rho$, and $\delta_u$ is the Dirac measure supported on $u$.)
\end{proof}

\begin{definition} \label{d:synmeasure} The {\it synchronized measure} is a $S,T$-invariant measure supported on synchronized configurations with the same mean spacing. We say that every configuration in the support of a synchronized measure is a {\it dynamic ground state}, and denote the set of all dynamic ground states with the mean spacing $\rho$ by $\mathcal{S}_{\rho}$.
\end{definition}

\begin{remark}
As discussed in \cite{Rabar15,Slijepcevic14b}, one can show that $\mathcal{S}_{\rho}$ is a closed, thus compact set. Not all synchronized configurations with the mean spacing $\rho$ are in $\mathcal{S}_{\rho}$, for example in the case of dynamics (\ref{r:ratchet1}) without the external force, i.e.\ $V' \equiv 0$, these are minimizing equilibria which are not spatially recurrent, i.e. with elementary discommensurations \cite{Aubry83}. In the depinned phase (see \cite{Rabar15,Slijepcevic14b} for details), configurations in $\mathcal{S}_{\rho}$ are in a certain sense globally attracting, and are always locally attracting. Details are omitted, as not required in the following.
\end{remark}

If $\mu$ is a synchronized measure with the mean spacing $\rho$, we define the transport speed of a measure $\hat{v}(\mu)$ as the expectation of the displacement over one period of the pulse:
\begin{equation}
\hat{v}(\mu) =\frac{1}{2\tau}  \int_{\mathcal{X}}((Tu)_0-u_0) d\mu. \label{r:vdef}
\end{equation}

Clearly by the $S$-invariance of $\mu$ and by the fact that $T,S$ commute, the site index $0$ in the definition of $\hat{v}(\mu)$ can be replaced by any $k \in \mathbb{Z}$. We now show that $\hat{v}$ depends only on the mean spacing $\rho$.

\begin{lemma} \label{l:speed1}
	Choose $\rho \in \mathbb{R}$, and a synchronized measure $\mu$ with the mean spacing $\rho$. Then for any $u \in \mathcal{S}_{\rho}$ and any $k \in \mathbb{Z}$,
	\begin{equation}
	\lim_{m \rightarrow \infty} \frac{1}{2m\tau}\left( (T^mu)_k-u_k \right) = \hat{v}(\mu). \label{r:vlimit}
	\end{equation}
	Furthermore, $\hat{v}(\mu)$ is the same for all the synchronized measures with the same mean spacing.
\end{lemma}

\begin{proof} We first prove the claim for $k = 0$. Choose $\rho \in \mathbb{R}$, and a synchronized measure $\mu$ with the mean spacing $\rho$. By the Birkhoff Ergodic Theorem applied to the function $u \mapsto (Tu)_0-u_0$, there exists $u$ in the support of $\mu$, thus in $\mathcal{S}_{\rho}$, such that the limit on the left-hand side of (\ref{r:vlimit}) for $k=0$ exists, denote it by $v_0$. 
	
Now if $w \in \mathcal{S}_{\rho}$, because of (\ref{r:width1}) we can find representations of $u,w$ in $\tilde{\mathcal{X}}$, denoted for simplicity again by $u,w$, so that $ u \leq w \leq u+4$.
By the order-preserving property, for any integer $m \geq 1$ we get $T^mw \geq T^mu$, thus $T^mw - w \geq T^mu-u-4$, and analogously $T^mw-w \leq T^mu-u + 4$. By taking $m \rightarrow \infty$, we see that $\lim_{m \rightarrow \infty}\frac{1}{2m\tau}\left( T^m(w)_0-w_0 \right)=v_0$, thus it is independent of the choice of $w \in \mathcal{S}_{\rho}$. By the Birkhoff ergodic theorem and (\ref{r:vdef}) we deduce that $v_0 = \hat{v}(\mu)$ which is (\ref{r:vlimit}) for $k=0$. As $\mu$ is $S$-invariant, this must hold for any $k\in \mathbb{Z}$. As $w$ was chosen arbitrarily in $\mathcal{S}_{\rho}$, we see that $\hat{v}(\mu)$ is the same for all $\mu$ supported on $\mathcal{S}_{\rho}$.
\end{proof}

By Lemma \ref{l:speed1}, we can define $v(\rho):= \hat{v}(\mu)$, where $\mu$ is any synchronized measure with the mean spacing $\rho$.

\begin{lemma} \label{l:speedbound}
	Let $u^0 \in \mathbb{R}^\mathbb{Z}$ be of bounded width with the mean spacing $\rho$ and constant $n$ as in (\ref{d:width}). Then for any $k,m \in \mathbb{Z}$, $m \geq 1$ and any $t_0 \in \mathbb{R}$, if $u(t)$ is the solution of (\ref{r:ratchet1}) or (\ref{r:ratchet2}), $u(t_0)=u^0$, we have that (\ref{r:five}) holds. 
\end{lemma}

\begin{proof} Without loss of generality let $t_0=0$.
Choose any $w \in \mathcal{S}_{\rho}$. It is easy to check that because of (\ref{d:width}) applied to $u^0$ and (\ref{r:width1}) applied to $w$, we can find representatives of $u^0$ and $w$ in $\tilde{\mathcal{X}}$ (for simplicity denoted by the same symbol), so that
\begin{equation} \label{r:est1}
 w \leq u^0 \leq w + n + 1.
\end{equation}
Fix an integer $m \geq 1$. As $w$ is synchronized, there exists an integer $j$ such that 
\begin{equation} \label{r:est2}
w + j \leq T^mw \leq w + j + 1.
\end{equation}
We claim that 
\begin{equation} \label{r:est3}
\frac{j}{2m\tau} \leq v(\rho) \leq \frac{j+1}{2m\tau}.
\end{equation}
Indeed, as in $\tilde{\mathcal{X}}$ we have that $T^m(u+j)=T^m(u)+j$, inductively from (\ref{r:est2}) we obtain that for any integer $r \geq 1$, 
$ rj \leq T^{rm}w - w \leq rj + r $. By taking $r \rightarrow \infty$ and using (\ref{r:vlimit}) and the definition of $v(\rho)$ we obtain (\ref{r:est3}). 

Combining (\ref{r:est1}) and (\ref{r:est2}) and using the order-preserving property, we see that $T^mu^0 \geq T^m w \geq w+ j \geq u^0 + j - n - 1$ and equivalently $T^mu^0 \leq T^mw + n + 1 \leq w + n + j + 2 \leq u^0 + n + j + 2$, thus
$$
 \frac{j-n-1}{2m\tau} \leq \frac{T^m u^0 - u^0}{2m\tau} \leq \frac{j+n+2}{2m\tau}.
$$
Combining it with (\ref{r:est3}) we obtain (\ref{r:five}).
\end{proof}

\begin{lemma} \label{l:continuous}
	The function $\rho \mapsto v(\rho)$ is continuous.
\end{lemma}

\begin{proof}
	As shown in \cite{Slijepcevic14b}, the set of all the synchronized measures is closed in the weak$^*$-topology. Let $\rho_k \in \mathbb{R}$ be a sequence converging to $\rho \in \mathbb{R}$, and choose a sequence of synchronized measures $\mu_k$ with the mean spacings $\rho_k$. It is easy to establish by using (\ref{d:width}) that all $\mu_k$ are supported on a compact subset of $\mathcal{X}$, thus without loss of generality assume $\mu_k$ is weak$^*$-convergent (otherwise we choose a convergent subsequence), converging to a synchronized measure $\mu$. As the mean spacing of a synchronized measure is the expectation of the continuous function $u \mapsto u_1-u_0$ on $\mathcal{X}$, we have that $\rho$ is the mean spacing of $\mu$. Also $\hat{v}$ is the expectation of the continuous function $u \mapsto \frac{1}{2 \tau}\left((Tu)_0-u_0 \right)$ on $\mathcal{X}$, thus $\hat{v}(\mu_k) \rightarrow \hat{v}(\mu)$ as $k \rightarrow \infty$, which completes the proof by the definition of $v(\rho)$.
\end{proof}

\begin{remark} Continuity of the analogue of the transport speed, the {\it average velocity}, was for related systems of equations proved by Wen-Xin Qin in \cite{Qin10,Qin13}, by another approach. The technique as in the proof of Lemma \ref{l:continuous}, by considering limits of measures, was introduced in \cite{Slijepcevic13}.
\end{remark}

\begin{remark}
 Consider a family of couplings $W$ or potentials $V$ continuously depending on a parameter $\lambda \in \mathbb{R}$. One can analogously to Lemma \ref{l:continuous} show that $v(\rho)$ changes continuously in $\lambda$. 
\end{remark}
         
\begin{proof}[Proof of Theorem \ref{t:speed}] Uniqueness of $v(\rho)$ follows from (\ref{r:five}) by taking $m \rightarrow \infty$. The other claims have been established in Lemmas \ref{l:speedbound} and \ref{l:continuous}.                                                                                                                                                                                                                                                                                                                                                                                                                                                                                                                                                                                                                                                                                                                                                   
\end{proof}

We close the section with a nice characterisation of the existence of transport.

\begin{corollary} \label{c:notzero}
	The transport speed $v(\rho)=0$ if and only if $T$ restricted to $\mathcal{S}_{\rho}$ is the identity. 
\end{corollary}

\begin{proof}
	The non-trivial implication is that $v(\rho)=0$ implies $Tu=u$ for all $u \in \mathcal{S}_{\rho}$. Assume $v(\rho)=0$, and choose a synchronized measure $\mu$ with the mean spacing $\rho$. Without loss of generality, assume $\mu$ is $S,T$-ergodic, as otherwise we can take any $S,T$-invariant measure in the $S,T$-ergodic decomposition of $\mu$ supported on $\mathcal{S}_{\rho}$. By a well-known characterization of ergodic measures for two commuting transformations on a compact set, we can find $u \in \mathcal{S}_{\rho}$ so that its $S,T$-orbit (i.e.\ the set of all $S^nT^m u$, $n,m \in \mathbb{Z}$) is dense in the support of $\mu$ (see \cite{Katok05}, Proposition 4.1.18, (2)). As $u$ is synchronized, $Tu \geq u$ or $Tu \leq u$. By the order-preserving property and continuity, we have that for all $w$ in the support of $\mu$, $Tw - w \leq 0$ or $Tw - w \geq 0$ and the sign is always the same (it depends only on $\rho$). Without loss of generality, assume $Tw \geq w$ on the support of $\mu$, thus $(Tw)_0-w_0 \geq 0$, $\mu$-a.e.. By definition, $v(\rho)=\hat{v}(\mu)$, thus 
	\[ 
	\int \left( (Tw)_0-w_0 \right)d\mu =0,
	\]
	so it must be $(Tw)_0=w_0$, $\mu$-a.e.. Now by continuity and $S$-invariance of $\mu$, we have $Tw=w$ on the support of $\mu$. As the union of supports of all $S,T$-ergodic measures is dense in $\mathcal{S}_{\rho}$, the claim must hold on $\mathcal{S}_{\rho}$.
\end{proof}

\begin{remark} \label{rm:hull}
We comment here on existence of a dynamical hull function and its relationship with $\mu(t)^*$. First, one can prove an analogue of the existence result for the dynamical hull function in the case of periodically driven Frenkel-Kontorova model, by Wen-Xin Qin \cite{Qin13} (see Theorem B), in the following sense. Consider the case when $\mu^0$ is a synchronized measure, which is $S$-ergodic, with the mean spacing $\rho = \hat{\rho}(\mu^0)$. Then one can prove analogously as in \cite{Qin13} existence of a two-variable dynamical hull function $H : \mathbb{R}^2 \rightarrow \mathbb{R}$, such that
\begin{equation}
H(x+1,y+1)=H(x+1,y) = H(x,y) + 1, \label{r:hull}
\end{equation}
and such that for every $t \in \mathbb{R}$, and for $\mu(t)$-a.e. $u$, we have that there exists $\alpha \in \mathbb{R}$ such that for all $n \in \mathbb{Z}$,
\begin{equation}
u_n = H\left( n \rho + 2v(\rho)\tau + \alpha, \frac{t}{2\tau} \right). \label{r:hull2}
\end{equation}
Fix $t$, and denote by $h_t$ the map on $\mathbb{S}^1$ induced by $x \mapsto H(x,t/(2\tau))$ ($h_t$ is well defined because of (\ref{r:hull})). Then $\mu(t)^*$ is the push-forward of the Lebesgue measure on $\mathbb{S}^1$ with respect to $h_t$. Here is why (we omit the details of the proof). By the standard arguments of the Aubry-Mather theory \cite{Bangert88}, one can show that for $\mu(t)$-a.e. $u$, $u_n \mapsto u_{n+1} \operatorname{mod} 1$ induces a circle homeomorphism, denoted by $f_t$. By construction, $\mu(t)^*$ is the invariant measure of $f_t$. Now (\ref{r:hull2}) can be rewritten as
$
 f_t = h_t \circ r_{\rho} \circ h_t^{-1},
$
where $r_{\rho}$ is the rotation of the circle $x \mapsto x +\rho \operatorname{mod} 1$, i.e. $h_t$ is the conjugacy between $r_{\rho}$ and $f_t$, thus the push-forward with respect to $h_t$ of the Lebesgue measure which is invariant for $r_{\rho}$ must be $\mu(t)^*$.
\end{remark}

\section{Dynamics in the off-phase} \label{s:five}

In this section we prove Proposition \ref{p:first}.
This follows from the following observations. If $\mu$ is a synchronized measure, then by the Birkhoff ergodic theorem and (\ref{r:width1}), as $\hat{\rho}(\mu)=\mathbb{E}_{\mu}[u_1-u_0]$, each $u$ in the support of $\mu$ has the mean spacing $\hat{\rho}(\mu)$. We show that decay of $\hat{\Delta}(\mu)^2=\mathbb{V}\text{ar}[u_1-u_0]$ in the off phase is proportional to the decay of the average coupling energy, which we define for a $S$-invariant measure $\mu$ with
\[
\hat{W}(\mu):=\int_{\mathcal{X}} W(u_1-u_0) d\mu.
\] 
We first establish that $D_{\text{off}}(\rho)=0$ (Lemma \ref{l:relax1}), then relate $\hat{\Delta}$ and $\hat{W}$ (Lemma \ref{l:relax2}), then prove a discrete-space analogue of the Poincar\'{e} inequality (Lemma \ref{l:relax3}), and finally combine it all by estimating the speed of decay of $\hat{W}(\mu(t))-W(\rho)$.

In this and the next section we fix a mean spacing $\rho \in \mathbb{R}$ and a synchronized measure $\mu=\mu(0)$ with the mean spacing $\rho$, and its evolution $\mu(t)$ with respect to (\ref{r:ratchet1}). We frequently use that for all $t \geq 0$, $\hat{\rho}(\mu(t))=\rho$, and that $\mu(t)$ is $S$-invariant.

The following lemma is equivalent to the fact used in \cite{Floria05}, that the average position of the site mod $1$ for a chosen rotationally ordered configuration does not change in the off-phase.

\begin{lemma} \label{l:relax1} 
	For any $k \in \mathbb{Z}$,
	$ \int_{\mathcal{X}}	\left( u_k(\tau)-u_k(0) \right) d\mu = 0$. Specifically, $D_{\text{off}}(\rho)=0$.
\end{lemma}
\begin{proof}
	By using the Fubini theorem, we get:
	\begin{align*}
		\int_{\mathcal{X}} \left( u_k(\tau)-u_k(0) \right) d\mu & =  \int_{\mathcal{X}} \int_0^{\tau} \frac{du_k(t)}{dt}\: dt \: d\mu 
		=  \int_0^{\tau} \left( \int_{\mathcal{X}} \frac{du_k(t)}{dt}\:d\mu \right )dt \\
		& =  \int_0^{\tau} \left( \int_{\mathcal{X}} W'(u_{k+1}(t)-u_k(t))d\mu - \int_{\mathcal{X}} W'(u_k(t)-u_{k-1}(t)) d\mu \right )dt \\
		& =  \int_0^{\tau} \left( \int_{\mathcal{X}} W'(u_{k+1}-u_k)d\mu(t) - \int_{\mathcal{X}} W'(u_k-u_{k-1}) d\mu(t) \right )dt,
	\end{align*}
	which is by the $S$-invariance of $\mu(t)$ equal to $0$.
\end{proof}

In the following proof we use the {\it width} function defined for any configuration of bounded width, $u \mapsto w_j(u)$ for $j \in \mathbb{Z}$ as follows:
$$ w_j(u)=\tilde{u}_j-\rho j -a_0,$$
where $\rho$ is the mean spacing of $u$, $\tilde{u}$ is a representative of $u$ in $\tilde{\mathcal{X}}$, and $a_0$ is the supremum of all $a\in \mathbb{R}$ such that for all $k \in \mathbb{Z}$, $\tilde{u}_k \geq k \rho + a$ (the definition is clearly independent of the choice of $\tilde{u}$). It is easy to check that by (\ref{r:width1}), for weakly rotationally ordered configurations we have that $|w_j(u)| \leq 1$, and that
\begin{equation}
w_j(u)-w_{j-1}(u)=u_j-u_{j-1}-\rho. \label{r:width2}
\end{equation}
Furthermore, it is straightforward to check that $w_j$ is upper semi-continuous, thus Borel measurable. 

We now establish a discrete-space version of the Poincar\'{e} inequality, bounding the $L^2$-norm of $u_1-u_0-\rho$ by the $L^2$-norm of its discrete-space derivative $u_1-2u_0+u_{-1}$.
\begin{lemma} \label{l:relax2}
	For all $t \geq 0$ we have that
	\begin{eqnarray}
	\hat{\Delta} (\mu(t))^2 \leq \left(  \int_{\mathcal{X}} (u_1-2u_0+u_{-1})^2 d\mu(t) \right)^{1/2}. \label{r:poincare}
	\end{eqnarray}
\end{lemma}
\begin{proof}
	First note the following identity holding for any sequences $z_j,w_j \in \mathbb{R}$, $j\in \mathbb{Z}$:
	$$ z_jw_j-z_{j-1}w_{j-1}=z_j (w_j-w_{j-1})+(z_j-z_{j-1})w_{j-1}.$$
	Now for some $u$ in the support of $\mu(t)$, thus rotationally ordered, we insert $z_j=(u_j-u_{j-1}-\rho)$ and $w_j=w_j(u)$. Applying (\ref{r:width2}) we easily get
	$$ (u_j-u_{j-1}-\rho)w_j(u)-(u_{j-1}-u_{j-2}-\rho)w_{j-1}(u)=(u_j-u_{j-1}-\rho)^2+(u_j-2u_{j-1}+u_{j-2})w_{j-1}(u).$$
	As all the terms are well defined functions on $\mathcal{X}$, we can integrate with respect to $\mu(t)$ for $j=1$. Note that because of the $S$-invariance of $\mu(t)$, the terms on the left-hand side vanish, so we have
	$$ \int_{\mathcal{X}} (u_1-u_0-\rho)^2 d\mu(t) = - \int_{\mathcal{X}} (u_1-2u_0+u_{-1})w_0(u) d\mu(t).$$
	It now suffices to apply the definition of $\hat{\Delta}(\mu(t))$ to the left-hand side, and the Cauchy-Schwartz inequality and $|w_0(u)| \leq 1$ to the right-hand side.
\end{proof}

\begin{lemma} \label{l:relax3} The following relations hold for all $t \geq 0$:
	\begin{gather}
	(\delta^-)^2 \hat{\Delta} (\mu(t))^4   \leq  \int_{\mathcal{X}} \left( W'(u_1-u_0)- W'(u_0-u_{-1}) \right)^2  d\mu(t),  \label{r:right} \\
		\frac{\delta^-}{2}  \hat{\Delta}(\mu(t))^2    \leq \hat{W}(\mu(t)) - W(\rho) \leq  \frac{\delta^+}{2}  \hat{\Delta}(\mu(t))^2. \label{r:leftright}
	\end{gather}
\end{lemma}

\begin{proof}
	The relation (\ref{r:right}) follows from the Mean Value Theorem applied to the function $W'$ and (\ref{r:poincare}).
	
	It suffices to prove the left-hand side of (\ref{r:leftright}), as the proof of the right-hand side is analogous.
	Let $x_1,...,x_n$ be any numbers in $[\rho-1,\rho+1]$, and let $\bar{x}=(x_1+...+x_n)/n$, for some integer $n>0$. By the second order Taylor theorem, we have
	\begin{equation*}
	W(x_k)-W(\bar{x}) \geq  W'(\bar{x})(x_k-\bar{x}) + \frac{\delta^-}{2}(x_k-\bar{x})^2,
	\end{equation*}
	thus by averaging we get
	\begin{equation}
	\frac{1}{n}\sum_{k=1}^nW(x_k) - W(\bar{x}) \geq  \frac{\delta^-}{2n} \sum_{k=1}^n (x_k-\bar{x})^2. \label{r:average}
	\end{equation}
	Now we insert $x_k=u_k-u_{k-1}$, note that then $\bar{x}=(u_n-u_0)/n$, and integrate (\ref{r:average}) with respect to $d\mu(t)$:
	\begin{eqnarray*}
	\frac{1}{n}\sum_{k=1}^n\int_{\mathcal{X}} W(u_k-u_{k-1})d\mu(t) - \int_{\mathcal{X}} W\left( \frac{u_n-u_0}{n} \right)d\mu(t)  \geq  \frac{\delta^-}{2n} \sum_{k=1}^n \int_{\mathcal{X}} \left(u_k-u_{k-1}-\frac{u_n-u_0}{n} \right)^2 d\mu(t). 
	\end{eqnarray*}
	By the $S$-invariance of $\mu(t)$ applied to the first term, and (\ref{r:width1}) applied to the right-hand side, we get
	\begin{align*}
	\hat{W}(\mu(t)) - \int_{\mathcal{X}} W\left( \frac{u_n-u_0}{n} \right)d\mu(t)  &\geq  \frac{\delta^-}{2n} \sum_{k=1}^n \int_{\mathcal{X}} \left(u_1-u_0-\frac{u_{n-k+1}-u_{1-k}}{n} \right)^2 d\mu(t) \\
	& = \frac{\delta^-}{2} \int_{\mathcal{X}} \left(u_1-u_0-\rho \right)^2 d\mu(t)  + O\left(\frac{1}{n}\right).
	\end{align*}
	Now it suffices to consider $n \rightarrow \infty$ and apply the Lebesgue Dominated Convergence Theorem to the second term, which thus converges to $W(\rho)$.
\end{proof}

We can now complete the proof, by using the fact noted in \cite{Slijepcevic00} that $\mu \rightarrow \hat{W}(\mu)$ is the Lyapunov function in the off-phase for the induced semiflow on the space of $S$-invariant measures.

\begin{proof}[Proof of Proposition \ref{p:first}] 
	We first express $d\hat{W}(\mu(t))/dt$, then use Lemmas \ref{l:relax2} and \ref{l:relax3} to obtain a differential inequality, and then we solve it. Denoting by $\dot{u}_k(t)=du_k(t)/dt$, we see that
	\begin{align}
		\frac{dW(u_1(t)-u_0(t))}{dt} & =  W'(u_1(t)-u_0(t))(\dot{u}_1(t)-\dot{u}_0(t)) \notag \\
		& =  W'(u_1(t)-u_0(t))\dot{u}_1(t)-W'(u_0(t)-u_{-1}(t))\dot{u}_0(t) \notag \\ & - \left(W'(u_1(t)-u_0(t))-W'(u_0(t)-u_{-1}(t))\right)\dot{u}_0(t). \label{r:justifiable}
	\end{align}
	Now integrating it with respect to $\mu(t)$, the first two summands on the right-hand side cancel out because of the $S$-invariance of $\mu(t)$. As $\dot{u}_0=W'(u_1(t)-u_0(t))-W'(u_0(t)-u_{-1}(t))$, we get
	\begin{align}
	\frac{d\hat{W}(\mu(t))}{dt} &= \frac{d}{dt} \int W(u_1-u_0)d\mu(t) = \int \frac{d}{dt} W(u_1-u_0)d\mu(t) \label{r:swap} \\
	 &=	 - \int \left( W'(u_1-u_0)-W'(u_0-u_{-1} \right)^2 d\mu(t). \notag
	\end{align}
	The permutation of the Lebesgue integral and derivative in (\ref{r:swap}) is justified by the Lebesgue dominated convergence theorem, applicable as the derivative $dW(u_1-u_0)/dt$ expressed by (\ref{r:justifiable}) is uniformly bounded on the support of $\mu(t)$.	
	
	We now apply (\ref{r:right}) and (\ref{r:leftright}) to the right-hand side, and get the differential inequality
	\begin{equation}
	\frac{d\hat{W}(\mu(t))}{dt} \leq - 4\left( \frac{\delta^-}{\delta^+}\right)^2 \left( \hat{W}(\mu(t)) - W(\rho) \right)^2. \label{r:diffineq}
	\end{equation}
	Note that by the left-hand side of (\ref{r:leftright}), we have
	$ \hat{W}(\mu(t)) \geq W(\rho)$.
	Unless $\mu(t)$ is for all $t$ supported on quasiperiodic configurations of the type $u_k=\rho \cdot k + a$ (in which case the dynamics in the off-phase is trivial and $\hat{\Delta}(\mu(\tau))=0$), we also deduce that $ \hat{W}(\mu(t)) > W(\rho)$. 
	Thus we can solve the differential inequality (\ref{r:diffineq}), and obtain
	$$ \hat{W}(\mu(\tau))-W(\rho) \leq \frac{1}{4\left( \frac{\delta^-}{\delta^+}\right)^2\tau + \left( \hat{W}(\mu(0))-W(\rho) \right)^{-1}}.$$
	Combining (\ref{r:width1}) and (\ref{r:leftright}) we get $\hat{W}(\mu(0))-W(\rho)\leq \delta^+/2$, thus
	$$ \hat{W}(\mu(\tau))-W(\rho)  \leq \frac{(\delta^+)^2}{4(\delta^-)^2\tau+2\delta^+},$$
	which combined with the left-hand side of (\ref{r:leftright}) completes the proof.
\end{proof}

\section{Dynamics in the on-phase} \label{s:six}

We focus now on the dynamics of a synchronized measure $\mu$ on the on-phase $t \in [\tau,2\tau]$, and prove Proposition \ref{p:second}. We first give an explicit lower bound on $D_{\text{on}}(\rho)$ in terms of $\mu(\tau)^*$, and then complete the proof.

Assume from now on in this section without loss of generality that (\ref{r:a2}) holds on the interval $[1/2-\alpha,1]$. Let $\varphi : \mathbb{S}^1 \rightarrow \mathbb{R}$ be defined with
\begin{equation}
\varphi(x)=\begin{cases} (\beta \tau) \wedge (1-x) & x > 1/2-\alpha, \\
-x & x \leq 1/2-\alpha, \end{cases} \nonumber
\end{equation}
where $\mathbb{S}^1$ is parametrized with $x\in [0,1)$.

\begin{lemma} \label{l:on3} Assume $\mu=\mu(0)$ is a synchronized measure. Then
	$$ D_{\text{on}}(\rho) \geq\frac{1}{2\tau} \int_{\mathbb{S}^1}\varphi(x)d\mu(\tau)^*.$$
\end{lemma}

\begin{proof} We first establish that for any weakly rotationally ordered $u$, we have that $|u_1-2u_0+u_{-1}| \leq 1$. Let $\tilde{u}$ be a representative of $u$ in $\tilde{\mathcal{X}}$, and find integer $n$ such that 
	\begin{equation}
	\tilde{u}_1 + n > \tilde{u}_0 \geq \tilde{u}_1 + n-1. \label{r:whyone}
	\end{equation} 
	Then as $\tilde{u}$ is weakly rotationally ordered, $S^{-1}\tilde{u}+n\geq \tilde{u}$, thus $\tilde{u}_{0} +n > \tilde{u}_{-1} $, so by summing it with the right-hand side of (\ref{r:whyone}) we obtain $\tilde{u}_1-2\tilde{u}_0+\tilde{u}_{-1} \leq 1$; $\tilde{u}_1-2\tilde{u}_0+\tilde{u}_{-1} \geq -1$ is analogous.
	Now take $u$ in the support of $\mu(t)$, thus rotationally ordered, such that $\pi(u) \in (1/2-\alpha,1]$. Then using (\ref{r:a2}) we get
	\begin{align*}
	\frac{d}{dt}u_{0}(t) &= -W'(u_0-u_{-1})+W'(u_1-u_0)+KV'(u_0) \\
	& \geq  - \delta^+ |u_1-2u_0+u_{-1}| + (\delta^+ + \beta) \geq  \beta.	
	\end{align*}
	Let $x=\pi(u(\tau))$. We easily deduce that, if $x \in [1/2-\alpha,1)$, then
	$$ u_0(2\tau)-u_0(\tau) \geq \beta \tau \wedge (1-x).$$
	Similarly, if $x \in [0,1/2-\alpha)$,
	$$ u_0(2\tau)-u_0(\tau) \geq - x.$$
	We thus see that
	$$
	\int (u_0(2\tau)-u_0(\tau) ) d \mu \geq  \int \varphi(\pi(u)) d\mu(u) 
	=  \int_{\mathbb{S}^1}\varphi(x)d\mu^*(\tau),
	$$
	which implies the claim by the definition of $D_{\text{on}}(\rho)$.
\end{proof}

\begin{proof}[Proof of Proposition \ref{p:second}] 
	Assume $d_1(\mu(\tau)^*,\lambda) \leq \ve \leq 1/4$.
	
	As $\varphi$ is piece-wise linear, and as it attains its minimum at $x=1/2-\alpha$, and its derivative is everywhere $\leq 1$, we can explicitly construct the probability measure $\nu$ on $\mathbb{S}^1$ for which $\int_{\mathbb{S}^1}\varphi(x) d\nu$ is minimal under the constraint $d_1(\nu,\lambda)\leq \ve$. Let $0 \leq \delta \leq 1$ be the unique number such that
	\begin{equation}
	\int_{1/2-\alpha}^{1/2-\alpha+\delta} d(x,1/2-\alpha) dx = \ve, \label{d:delta}
	\end{equation}
	where $d(x,y)$ is the distance on $\mathbb{S}^1$. 
	The measure $\nu$ satisfies $d_1(\nu,\lambda) \leq \ve$, and is constructed by "transporting" the measure $\lambda$ from $ x \in \mathbb{S}^1$ to $x=1/2-\alpha$, starting from points for which $\varphi(x)-\varphi(1/2-\alpha)$ is larger (here "transport" is used in the sense of minimal transport characterisation of the Wasserstein distance). More precisely, let
	$$
	\nu = \delta \cdot \delta_{ 1/2 - \alpha } + \lambda_{[1/2-\alpha,1/2-\alpha + \delta)^c}
	$$
	(where $\delta_{x}$ is the Dirac measure, $\lambda_A$ the Lebesgue measure on the interval $A$, and the complement and calculations are on $\mathbb{S}^1$). The coupling $\zeta$ of $(\nu,\lambda)$ is given with
	$ \zeta =  \delta_{ 1/2 - \alpha } \times \lambda_{[1/2-\alpha,1/2-\alpha + \delta)} + \zeta_0$, $\zeta_0$ the "diagonal" measure uniformly supported on $\lbrace (x,x), \: x\in [1/2-\alpha,1/2-\alpha + \delta)^c \rbrace$,
	which gives $d_1(\nu,\lambda) \leq \ve$ because of (\ref{d:delta}) and the definition of the $L^1$-Wasserstein distance. 
	
	Consider the case $0 \leq \delta \leq 1/2+\alpha -\beta \tau$. Then by Lemma \ref{l:on3},
	\begin{align}
	2\tau D_{\text{on}}(\rho)  & \geq   \int_{\mathbb{S}^1} \varphi(x) d\nu   \nonumber \\
	& =   \int_0^{1/2-\alpha} (-x)dx - \delta(1/2-\alpha) + \int_{1/2-\alpha+\delta}^{1-\beta\tau}\beta \tau dx + \int_{1-\beta\tau}^1 (1-x)dx. \notag
	\end{align}
	Careful evaluation yields
	\begin{equation}
	2\tau D_{\text{on}}(\rho) \geq \alpha -\delta+\frac{1}{2}\delta^2-\frac{1}{2}\left( \frac{1}{2} + \alpha - \beta \tau - \delta \right)^2. \label{r:final2}
	\end{equation}
	Similarly, if $ 1/2+\alpha -\beta \tau \leq \delta \leq \alpha + 1/2$,
	\begin{align}
	2\tau D_{\text{on}}(\rho)  & \geq  \int_{\mathbb{S}^1} \varphi(x) d\nu  
	=   \int_0^{1/2-\alpha} (-x)dx - \delta(1/2-\alpha) + \int_{1/2-\alpha+\delta}^1 (1-x)dx \notag \\
	& =  \alpha - \delta(1/2-\alpha) - \int_{1/2-\alpha}^{1/2-\alpha+\delta} (1-x)dx \notag \\
	& =  \alpha- \delta + \frac{1}{2}\delta^2. \label{r:final3}
	\end{align}
	Finally, if $\alpha + 1/2 \leq \delta \leq 1$, we easily get that
	\begin{equation}
	2\tau D_{\text{on}}(\rho)   \geq  \int_{\mathbb{S}^1} \varphi(x) d\nu \geq   \int_{\delta-1/2-\alpha}^{1/2-\alpha} (-x)dx - \delta(1/2-\alpha) = \alpha- \delta + \frac{1}{2}\delta^2. \label{r:final3b}
	\end{equation}
	From (\ref{d:delta}) we obtain
	\begin{equation*}
	\ve=\begin{cases} \delta^2/2 &  \delta \leq 1/2 \\
	\delta-\delta^2/2 -1/4 &  1/2 \leq \delta \leq 1,
	\end{cases} 
	\end{equation*}
	thus it is easy to check that $\delta \leq 2 \ve^{1/2}$.
	Inserting that in (\ref{r:final2}), (\ref{r:final3}) and (\ref{r:final3b}), as in all three cases the right-hand sides are decreasing functions in $\delta$, we complete the proof.
\end{proof}

\section{Standard deviation of the spatial displacement and Wasserstein distance} \label{s:seven}

We now prove Proposition \ref{p:third}. In this section, $\mu$ is an arbitrary $S$-invariant probability measure on $\mathcal{X}$, such that $\hat{\rho}(\mu)$ and $\hat{\Delta}(\mu)$ are finite. For some $u \in \mathcal{X}$, denote by $\xi_q^*(u)$ the probability measure on $\mathbb{S}^1$ uniformly supported on $u_0,u_1,...,u_{q-1} \mod 1$.

\begin{lemma} \label{l:on1}
For any integer $q \geq 1$,
	\begin{equation}
	    d_1(\mu^*,\lambda) \leq \int_{\mathcal{X}} d_1(\xi^*_q(u),\lambda) d\mu(u). \label{r:claim0}
	\end{equation}
\end{lemma}

\begin{proof} We prove it by approximating $\mu$ with a sequence of measures supported on finite sets, proving the claim for these measures, and then extending the inequality by continuity of the Wasserstein distance.
	
Recall \cite[Theorem 6.3]{Parathsarathy:67} that on separable metric spaces, the set of probability measures uniformly supported on finite sets is dense in weak$^*$-topology. (The claim in \cite[Theorem 6.3]{Parathsarathy:67} is that the set of probability measures supported on finite sets in dense; it is easy to extend the statement to the set of probability measures {\it uniformly} supported on finite sets.) Thus on separable metric space $\mathcal{X}$, we can find a sequence of probability measures $\tilde{\mu}_n$ uniformly supported on finite sets $\lbrace u^{1,n},....,u^{n,n}\rbrace \subset \mathcal{X}$, i.e. defined with $\tilde{\mu}_n=\frac{1}{n}\sum_{k=1}^n\delta_{u^{k,n}}$, where $\delta_{w}$ is the Dirac measure; such that $\tilde{\mu}_n$ converges in weak$^*$-topology to $\mu$. Now by the definition of $\tilde{\mu}_n$,
\begin{equation}
  \frac{1}{n}\sum_{k=1}^n d_1(\xi_q^*(u^{k,n}), \lambda) = \int_{\mathcal{X}} d_1(\xi^*_q(u),\lambda) d\tilde{\mu}_n(u).
\end{equation}
Fix $q \geq 1$, and let $\mu_n$ be uniformly supported on 
$
\lbrace S^{-j}u^{k,n}, \, 1 \leq j \leq q, \, 1 \leq k \leq n \rbrace.
$
Because of $S$-invariance of $\mu$, we deduce that $\mu_n$ also converges in weak$^*$-topology to $\mu$. By continuity of $\pi$, we see that $\mu_n^* \rightarrow \mu^*$. As the $L^1$-Wasserstein distance generates the weak$^*$-topology \cite[Theorem 7.12]{Villani:08}, $\nu \mapsto d_1(\nu,\lambda)$ is weak$^*$-continuous, thus 
	\begin{equation}
	\lim_{n \rightarrow \infty} d_1(\mu_n^*,\lambda)=d_1(\mu^*,\lambda). \label{r:Wfirst}
	\end{equation} Similarly we show that $u \mapsto d_1(\xi^*_q(u),\lambda)$ is continuous as a function $\mathcal{X} \rightarrow \mathbb{R}$, thus by definition of weak$^*$-convergence,  
		\begin{equation}
		\lim_{n \rightarrow \infty} \int_{\mathcal{X}} d_1(\xi^*_q(u),\lambda) d\tilde{\mu}_n(u)= \int_{\mathcal{X}} d_1(\xi^*_q(u),\lambda) d\mu(u).
		\label{r:Wsecond}
		\end{equation}
Finally, we claim the following:
    \begin{equation}
      d_1(\mu_n^*,\lambda) \leq \frac{1}{n}\sum_{k=1}^n d_1(\xi_q^*(u^{k,n}),\lambda). \label{r:claim}
    \end{equation}	
Indeed by construction, $\mu_n^* = \frac{1}{n}\sum_{k=1}^n \xi_q^*(u^{k,n})$, thus if $\zeta_1$,...,$\zeta_n$ are couplings of $(\xi_q^*(u^{1,n}),\lambda)$,..., $(\xi_q^*(u^{n,n}),\lambda)$, then $\zeta = \frac{1}{n}\sum_{k=1}^n\zeta_k$ is a coupling of $(\mu_n^*,\lambda)$, which by definition of the $L^1$-Wasserstein distance implies (\ref{r:claim}). Combining (\ref{r:claim0}) and (\ref{r:claim}) we obtain  
	$$
	d_1(\mu_n^*,\lambda) \leq \int_{\mathcal{X}} d_1(\xi^*_q(u),\lambda) d\tilde{\mu}_n(u),
	$$
which in the limit $n \rightarrow \infty$ gives the claim by (\ref{r:Wfirst}) and (\ref{r:Wsecond}).
\end{proof}

\begin{proof}[Proof of Proposition \ref{p:third}]  For $u \in \mathcal{X}$, an integer $q \geq 1$ and $x \in \mathbb{R}$, we define $\nu^*_{x,q}(u)$ to be the probability measure on $\mathbb{S}^1$, uniformly supported on the set
\begin{align*}
  u_{(q-1)/2}+ k x \mod 1, & \quad k=-\frac{q-1}{2},...,\frac{q-1}{2},   \hspace{5ex} q \text{ odd}, \\
  \frac{1}{2}\left( u_{q/2-1}+u_{q/2} \right) + \left(k+\frac{1}{2}\right)x \mod 1, & \quad  k=-\frac{q}{2},...,\frac{q}{2}-1, \hspace{9ex} q \text{ even}
\end{align*} (the somewhat more involved construction of $\nu^*_{x,q}(u)$ in the even case is required to obtain the same bound as in the odd case). Choose any integer $p \neq 0$ relatively prime with $q$.
We will estimate $d_1(\nu^*_{p/q,q}(u),\lambda)$, $d_1(\nu^*_{\rho,q}(u),\nu^*_{p/q,q}(u))$, $d_1(\xi^*_q(u),\nu^*_{\rho,q}(u))$ for some $u \in \mathcal{X}$ and $\rho = \hat{\rho}(\mu)$, then apply the triangle inequality and Lemma \ref{l:on1}. 

We first establish that
\begin{equation}
d_1\left( \nu^*_{p/q,q}(u),\lambda \right) \leq \frac{1}{4q}. \label{r:d1a}
\end{equation}
Indeed, denote the support of $\nu^*_{p/q,q}$ by $\tilde{u}_0,...,\tilde{u}_{q-1} \in \mathbb{S}^1$, and let $\lambda_{[a,b]}$ be the Lebesgue measure on $[a,b] \subset \mathbb{S}^1$. Consider the coupling $\zeta$ of $(\nu^*_{p/q,q},\lambda)$ given with
$$
\zeta = \sum_{k=0}^{q-1} \delta_{\tilde{u}_k} \times \lambda_{[\tilde{u}_k-1/(2q),\tilde{u}_k+1/(2q)]}.
$$
It is easy to check that $\int_{\mathbb{S}^1 \times \mathbb{S}^1} |x-y| d\zeta(x,y) = 1/(4q),$ which gives (\ref{r:d1a}).

For odd $q$ we can deduce	
	\begin{equation}
		d_1(\nu^*_{\rho,q}(u), \nu^*_{p/q,q}(u))  \leq \frac{1}{q} \sum_{k=-(q-1)/2}^{(q-1)/2} |k| \, |p/q-\rho| \leq \frac{q^2-1}{4q}|p/q-\rho| \leq \frac{q}{4}|p/q-\rho|.  \label{r:d1b}
	\end{equation}		
We can see that as follows: note that $\nu^*_{\rho,q}(u)$ and $\nu^*_{p/q,q}(u)$ are measures uniformly supported on sets with $q$ elements. We now obtain the first inequality in (\ref{r:d1b}) by constructing a coupling by pairing $u_{(q-1)/2}+ k x$ with the same $k$ for $x= \rho$ and $x=p/q$ respectively. Similarly we obtain the same bound for even $q$.

Let 
	$$ \Delta_q(u)=\left(\frac{1}{q} \sum_{k=0}^{q-1}(u_{k+1}-u_k-\rho)^2 \right)^{1/2}.$$
Consider first odd $q$. Then we obtain an upper bound on $d_1(\xi^*_q(u),\nu^*_{\rho,q}(u))$ by considering the coupling concentrated on 
$$\left\lbrace \left(u_{k+(q-1)/2} \mod 1,u_{(q-1)/2}+k\rho  \mod 1 \right), \, k=-\frac{q-1}{2},...,\frac{q-1}{2} \right\rbrace $$
with equal weights $1/q$. Thus by telescoping and then applying the Cauchy-Schwartz inequality,
	\begin{align}
	d_1(\xi^*_q(u),\nu^*_{\rho,q}(u))  & \leq \frac{1}{q} \sum_{k=-(q-1)/2}^{(q-1)/2} \left|u_{k+(q-1)/2}-u_{(q-1)/2}-k\rho  \right| \notag \\
	& \leq \frac{1}{q} \sum_{k=-(q-1)/2}^{-1} \left| (q-1)/2+1 + k \right|  
		 \left| u_{(q-1)/2+k+1}-u_{(q-1)/2+k} - \rho \right|  \notag \\
	& \quad + \frac{1}{q} \sum_{k=1}^{(q-1)/2} \left|(q-1)/2+1 - k \right| \, \left| u_{(q-1)/2+k}-u_{(q-1)/2+k-1} - \rho \right| \notag \\
	& \leq \frac{1}{q} \left( 2 \sum_{k=1}^{(q-1)/2}k^2 \right)^{1/2} (q-1)^{1/2} \Delta_{q-1}(u) \leq \left( \frac{(q^2-1)(q-1)}{12q} \right)^{1/2} \Delta_{q-1}(u) \nonumber \\
	& \leq \frac{q}{\sqrt{12}} \Delta_{q-1}(u). \label{r:d1c}
	\end{align}
Analogously for even $q$, we obtain the same conclusion (details are routine, thus omitted):
\begin{equation*}
	d_1(\xi^*_q(u),\nu^*_{\rho,q}(u)) \leq \left( \frac{(q^2+2)(q-1)}{12q} \right)^{1/2} \Delta_{q-1}(u) \leq \frac{q}{\sqrt{12}} \Delta_{q-1}(u). 
\end{equation*}
We sum (\ref{r:d1a}), (\ref{r:d1b}), (\ref{r:d1c}) and apply the triangle inequality, thus
\begin{equation}
d_1(\xi^*_q(u),\lambda) \leq \frac{1}{4q} + \frac{q}{4}|p/q-\rho| + \frac{q}{\sqrt{12}} \Delta_{q-1}(u). \label{r:d1d}
\end{equation}
By the inequality between the arithmetic and quadratic mean applied to $u \mapsto \Delta_{q-1}(u)$, and by $S$-invariance of $\mu$, we get
\begin{align}
\int_{\mathcal{X}} \Delta_{q-1}(u) d\mu(u) & \leq \left( \int_{\mathcal{X}} \Delta_{q-1}(u)^2 d\mu(u) \right)^{1/2} 
 = \left( \frac{1}{q-1}\int_{\mathcal{X}} \sum_{k=0}^{q-2}(u_{k+1}-u_k-\rho)^2 d\mu(u) \right)^{1/2} \notag \\
& = \left(\int_{\mathcal{X}} (u_1-u_0-\rho)^2 d\mu(u) \right) = \hat{\Delta}(\mu). \label{r:d1e}
\end{align}
It suffices now to insert (\ref{r:d1d}) and (\ref{r:d1e}) in Lemma \ref{l:on1} to obtain
\begin{equation}
 d_1(\mu^*,\lambda) \leq \frac{1}{4q} + \frac{q}{4}|p/q-\rho| + \frac{q}{\sqrt{12}} \hat{\Delta}(\mu). \label{r:lastlast}
\end{equation}
It remains to show that $ d_1(\mu^*,\lambda) \leq 1/4$. This, however, holds for any probability measure $\mu^*$ on $\mathbb{S}^1$, where the upper bound is obtained for the coupling $\zeta=\mu^* \times \lambda$.
\end{proof}

\section{Proof of Corollary \ref{c:c}} \label{s:last}

We now prove Corollary \ref{c:c}. Consider first rational mean spacings $\rho = p_0/q_0$, $p_0,q_0$ relatively prime, $q_0 \geq 1$. Then from (\ref{d:d}) one can deduce that 
\begin{equation}
\lim_{\tau \rightarrow \infty} d(p_0/q_0,\tau) = \frac{1}{4q_0}. \label{r:limlim}
\end{equation}
Indeed, it is a simple number-theoretical argument that $d(p_0/q_0,\tau) \geq 1/(4q_0)$. The equality in (\ref{r:limlim}) is then obtained by taking $p/q = p_0/q_0$ in (\ref{d:d}). Inserting (\ref{r:limlim}) in (\ref{r:mainpotential}), and noting that the last term in (\ref{r:mainpotential}) is $0$ for $\tau$ large enough, completes (i).

For the remaining claims, we use the lower bound (\ref{r:mainpotential}) with $\tilde{d}(\rho,\tau)$ inserted in (\ref{r:mainpotential}) instead of $d(\rho,\tau)$, in the light of Remark \ref{r:tilded}. Consider irrational $\rho$, and its associated sequence of convergents $p_n/q_n$, $p_n,q_n$ relatively prime, $q_n \geq 1$. From (\ref{r:td}) one easily sees that there exists an increasing sequence of $\tau_n$ such that for $\tau \geq \tau_n$ we have $\tilde{d}(\rho,\tau) \leq 1 / q_n$. Inserting it in (\ref{r:mainpotential}), we get that for $n$ large enough and $\tau \geq \tau_n$, the right-hand side of (\ref{r:mainpotential}) is $>0$, thus transport exists.

We first simplify the expression for $\tilde{d}(\rho,\tau)$, using $c$ as in Corollary \ref{c:c}, noting that
\begin{equation}
\frac{1}{2q_n} + \frac{q_n}{\sqrt{12}}\cdot \frac{\delta^+}{\left( 2(\delta^-)^3\tau+\delta^+ \delta^- \right)^{1/2}} = \frac{1}{2q_n} + \frac{q_n \, c^4}{2}\tau^{-1/2} + q_n\mathcal{O}\left(\tau^{-3/2}\right). \label{r:asym1}
\end{equation}

Inserting now $\tau_n = c^8 q_n^4$, we see that
$$
\tilde{d}(\rho,\tau_n) = \frac{1}{q_n} + \mathcal{O}(\tau_n^{-1}) = c^2 \tau_n^{-1/4} + \mathcal{O}\left(\tau_n^{-5/4}\right),
$$ 
thus
\begin{equation*}
\tilde{d}(\rho,\tau_n)^{1/2} = c \tau_n^{-1/8} + \mathcal{O}\left(\tau_n^{-9/8}\right). 
\end{equation*}
Inserting it and $\tilde{d}
(\rho,\tau_n)=\mathcal{O}\left(\tau_n^{-1/4}  \right)$ into (\ref{r:mainpotential}) we obtain
\begin{equation}
v(\rho,\tau_n) \geq \frac{\alpha}{2}\tau_n^{-1} - c \tau_n^{-9/8} + \mathcal{O}\left(\tau_n^{-5/4}\right), \label{r:rhotau}
\end{equation}
which completes (ii).

To prove (iii), recall the Khinchin-L\'{e}vy Theorem \cite{Khinchin97}, which establishes that there is a set $R \subset \mathbb{R}$ of full Lebesgue measure such that for $\rho \in R$, the sequence of denominators $q_n$ of convergents of $\rho$ satisfies
\begin{equation}
\lim_{n \rightarrow \infty} q_n^{1/n} = \gamma_L,
\label{r:khle}
\end{equation} 
where
$ \gamma_L=\exp (\pi^2 / (12 \ln 2))$ is the L\'{e}vy's constant. Thus for any $\varepsilon > 0$ and $n \geq n(\varepsilon)$ for some sufficiently large $n(\varepsilon)$ (also depending on $\rho$), $(1-\varepsilon) \gamma_L^n \leq q_n \leq (1+\varepsilon) \gamma_L^n$. For $\tau$ large enough, find $n$ so that
$$
(1 + \varepsilon) \gamma_L^n \leq c^{-2}\tau^{1/4} \leq (1 + \varepsilon) \gamma_L^{n+1},
$$
and then 
$$
  c^{-2}\tau^{1/4} \frac{1-\varepsilon}{\gamma_L(1+\varepsilon)} \leq q_n \leq   c^{-2} \tau^{1/4}.
$$
Inserting that in the definition of $d(\rho,\tau)$, we obtain
\begin{align*}
d(\rho,\tau)^{1/2} & \leq c \left( \frac{\gamma_L(1 + \varepsilon)}{2(1-\varepsilon)} + \frac{1}{2} \right)^{1/2}\tau^{-1/8} + \mathcal{O}\left(\tau^{-9/8} \right) 
\\
& = c \left( \frac{\gamma_L+1}{2} \right)^{1/2}(1+{\varepsilon}) \tau^{-1/8}  + \mathcal{O}\left(\tau^{-9/8}\right),
\end{align*}
where in the second row we redefined $\varepsilon$ without changing the notation. It suffices now to insert that and $d(\rho,\tau_n)=\mathcal{O}\left(\tau_n^{-1/4}  \right)$ into (\ref{r:mainpotential}) complete (iii).

\begin{remark} \label{r:last}
We comment on estimating the error term in (\ref{r:expansion}). As follows from the calculation above, the required missing information is how quickly the Khinchin-L\'{e}vy limes (\ref{r:khle}) converges. 

As none of the results in the paper depend on the integer part of $\rho$, it suffices to consider $\tilde{R}=R \cap [0,1]$ and mean spacings $\rho \in \tilde{R}$. Then by a variant of the Central Limit Theorem as e.g. in \cite{Misyavichyus83}, one can find an increasing sequence of sets $R_n$, $\cup_{n \in \mathbb{N}} R_n = \tilde{R}$ with their Lebesgue measures approaching $1$, and explicit estimates of $|q_n^{1/n}-\gamma_L|$ for all $\rho \in R_n$ and its series of convergents $p_n/q_n$. This would lead to an explicit estimate of the constant next to the error term in (\ref{r:expansion}) for all $\rho \in R_n$; we omit the details.
\end{remark}

\section{Discussion}

As noted in Example \ref{e}, the lower bound (\ref{r:mainpotential}) is inefficient and seems rather far from optimal. The main loss in our approach and opportunity for improvement seems to be in estimates for the on-phase, and in particular in the bound on the interaction force derived solely from the inequality $|u_1-2u_0+u_{-1}| \leq 1$ at the beginning of the proof of Lemma \ref{l:on3}. This bound follows from rather general order-preserving property of the dynamics only, and can likely be improved by e.g. complementing it with energy estimates containing further information on the dynamics.

Related to that, the results extend to second order dynamics (i.e. with inertial effects), as long as the dynamics remains overdamped in the sense of \cite{Baesens:04}, thus the order-preserving property holds. Furthermore, they are very likely stable for perturbations of piecewise constant pulses belonging to a larger class of pulses, with $v(\rho)$ likely varying continuously, as shown in \cite{Qin:18} for forced systems of the same type.

The measure-theoretical approach seems to be a promising tool to address analogous or similar questions for higher dimensional chains, for chains with non-convex interactions, and with inertial dynamics with small or no damping. As in all these cases the order preserving property of the dynamics fails, the critical remaining work is to establish bounds in the on-phase by energy estimates, or by another technique.

\section*{Acknowledgement}

We thank A. Dujella for help with facts related to continued fractions, Wen Xin-Qin for inspiring discussions, and anonymous referees for insightful comments and suggestions. This work was supported by the Croatian Science Foundation grant No IP-2014-09-2285.

\end{document}